\documentclass[11pt,a4paper,review, final]{siamart171218}
\usepackage[utf8]{inputenc}
\usepackage[english]{babel}
 \UseRawInputEncoding
 
\usepackage[ruled,vlined,algo2e]{algorithm2e}
\usepackage{graphicx}
\usepackage{fancyhdr}

\usepackage{amsfonts}
\usepackage{amssymb}
\usepackage{bm}

\usepackage[normalem]{ulem}
\usepackage{amsmath}
\usepackage{amsfonts}
\usepackage{amssymb}
\usepackage{mathtools}
\usepackage{enumitem}
\usepackage[percent]{overpic}
\usepackage{tikz}
\usepackage{mathrsfs}
\usepackage{xargs}
\usepackage{xcolor}
\usepackage[colorinlistoftodos,prependcaption,textsize=tiny]{todonotes}
\newcommandx{\at}[2][1=]{\todo[linecolor=red,backgroundcolor=red!25,bordercolor=red,#1]{#2}}
\usepackage[font=small,skip=5pt]{caption}

\usepackage{mathtools,booktabs}
\DeclarePairedDelimiter\ceil{\lceil}{\rceil}

\usepackage{varwidth}
\usepackage{hyperref}
\usepackage{cleveref}
\usepackage{cite}
\usepackage{comment}

\def\AA{A^T\!A}
\def\BB{B^T\!B}
\def\RR{R^T\!R}
\def\YY{Y^T\!Y}

\DeclareMathOperator*{\argmin}{\arg\!\min}

\DeclareMathOperator{\R}{\mathbb{R}}

\newcommand{\ignore}[1]{}

\graphicspath{ {./Graphics/} }

\title{Randomized algorithms for Tikhonov regularization in Linear Least Squares
}
\author{Maike Meier \and Yuji Nakatsukasa\thanks{Mathematical Institute, University of Oxford, Oxford, OX2 6GG, UK. (\email{meier@maths.ox.ac.uk, nakatsukasa@maths.ox.ac.uk})}. The first author is supported by the Oxford-Wang Graduate Scholarship.}
\headers{Randomized algorithms for Tikhonov}{Maike Meier and Yuji Nakatsukasa}
\begin{document}
\maketitle

\begin{abstract}
We describe two algorithms to efficiently solve regularized linear least squares systems based on sketching.  The algorithms compute preconditioners for $\min \|Ax-b\|^2_2 + \lambda \|x\|^2_2$, where $A\in\R^{m\times n}$ and $\lambda>0$ is a regularization parameter, such that LSQR converges in $\mathcal{O}(\log(1/\epsilon))$ iterations for $\epsilon$ accuracy. We focus on the context where the optimal regularization parameter is unknown, and the system must be solved for a number of parameters $\lambda$. Our algorithms are applicable in both the underdetermined $m\ll n$ and the overdetermined $m\gg n$ setting.  Firstly, we propose a Cholesky-based sketch-to-precondition algorithm that uses a `partly exact' sketch, and only requires one sketch for a set of $N$ regularization parameters $\lambda$.  The complexity of solving for $N$ parameters is $\mathcal{O}(mn\log(\max(m,n)) +N(\min(m,n)^3 + mn\log(1/\epsilon)))$.  Secondly, we introduce an algorithm that uses a sketch of size $\mathcal{O}(\text{sd}_{\lambda}(A))$ for the case where the statistical dimension $\text{sd}_{\lambda}(A)\ll\min(m,n)$. The scheme we propose does not require the computation of the Gram matrix, resulting in a more stable scheme than existing algorithms in this context.  We can solve for $N$ values of $\lambda_i$ in  $\mathcal{O}(mn\log(\max(m,n)) + \min(m,n)\,\text{sd}_{\min\lambda_i}(A)^2 + Nmn\log(1/\epsilon))$ operations. 
\end{abstract}



\section{Introduction}\label{sec:intro}
Tikhonov regularization is a regularization technique for linear least squares (LLS) problems. Consider the LLS problem
\begin{equation}\label{eq:LLs}
	\min_{x\in\R^n}\|Ax - b\|_2^2,
\end{equation}
where $A\in\R^{m\times n}$ and $b\in\R^{m}$.  In this paper we consider both the case $m \gg n$ (overdetermined) and $n \gg m$ (underdetermined). 

If the design matrix $A$ in \eqref{eq:LLs} is ill-conditioned, or essentially always in the underdetermined case, it may be necessary to regularize the problem before solving it numerically. In particular, the underdetermined problem without regularization leads to non-unique solutions and ill-conditioning in the design matrix could result in numerical errors or excessive computing times.  More importantly, the ill-conditioning can result in amplification of the (e.g. measurement or approximation) error in $b$ (and $A$)~\cite{Cohen2013}.
The most common form of regularization, Tikhonov regularization (also known as ridge regression), transforms the LLS problem \eqref{eq:LLs} to
\begin{equation}\label{eq:regLLS}
	\min_{x\in\R^n}\|Ax - b\|^2_2 + \lambda \|x\|_2^2,
\end{equation}
for a regularization parameter $\lambda > 0$ \cite{Bjorck1996NumericalProblems}. We denote the minimizer to \eqref{eq:regLLS} for a particular $\lambda$ by $x_{\lambda}$. 

For overdetermined problems we can transform the regularized problem to a standard LLS formulation by considering
\begin{equation}\label{eq:regLLSover}
    \min_{x\in\R^{n}}\|Bx - \hat{b}\|_2^2,\quad B = \begin{bmatrix}A \\ \sqrt{\lambda}I_n\end{bmatrix}, \quad \hat{b} = \begin{bmatrix}b \\ 0 \end{bmatrix},
\end{equation}
where $B\in\R^{(m+n)\times n}$ and $b\in\R^{m+n}$. Similarly, for the underdetermined problem we can equivalently to \eqref{eq:regLLS} find the minimum-norm solution to the problem~\cite{Cohen2013}
\begin{equation}\label{eq:regLLSunder}
     \min_{x\in\R^n,\,y\in\R^m}\left\|D\begin{bmatrix}
    x\\ y 
    \end{bmatrix} - b\right\|^2_2, \quad D = \begin{bmatrix}
    A & \sqrt{\lambda}I_m
    \end{bmatrix}.
\end{equation}
The solution $x_{\lambda}$ to \eqref{eq:regLLS} is then given by the top $n$ coordinates $x$ of the solution that minimizes \eqref{eq:regLLSunder}. 

The conditioning of the problem improves for larger $\lambda$ (while the norm $\|x\|$ of the solution for \eqref{eq:regLLS} decreases), and $x$ tends to the solution of \eqref{eq:LLs} as $\lambda$ tends to zero. The optimal regularization parameter is usually unknown a priori and is to be determined in an ad hoc manner, for instance by considering an L-curve~\cite{Hansen2001TheProblems}. This involves solving \eqref{eq:regLLS} for a number of regularization parameters $\lambda_1,\dots,\lambda_N$. We consider this context and aim to solve for this number of parameters efficiently.

In this work we build on the vast body of previous
 work on randomized numerical linear algebra for (regularized) least squares problems. In particular, we present algorithms to design preconditioners for the problems \eqref{eq:regLLSover} and \eqref{eq:regLLSunder} based on sketching.  Sketching is a method that multiplies a matrix with a smaller random matrix to obtain a lower-dimensional matrix that preserves as much information of the original matrix as possible.  The sketch-to-precondition framework was made widely known through the work of Rokhlin and Tygert in 2008~\cite{Rokhlin2008}. A fast implementation was later discussed in \cite{Avron2010a}. The main idea is to sketch \cite{Woodruff2014} the design matrix $A$ and use the information from the sketch to find a preconditioner.  Our algorithms are also based on this idea, but particularly designed for the context of Tikhonov regularization.

The first algorithm we introduce is closely related to the classic sketch-to-precondition work. The two main distinguishing contributions are 1) the use of a `partly exact'~\cite{Avron2017} sketch specific to Tikhonov regularization and 2) the use of the more efficient Cholesky decomposition instead of the QR decomposition or the singular value decomposition.  The Cholesky decomposition is especially efficient in the regularization context because for each additional value of $\lambda$ there is less work involved as compared to the QR decomposition.  This will be discussed in more detail in the next section. 

We secondly introduce an algorithm that can efficiently tackle the situation where the matrix $A$ has rapidly decaying singular values.  The quantity of interest in this case,  in the context of Tikhonov regularization, is the \emph{statistical dimension} (or degrees of freedom).
\begin{definition}[Statistical dimension]\label{def:statdim} For $\lambda\geq 0$ and a rank-$k$ matrix $A$ with singular values $\sigma_1(A),\dots,\sigma_k(A)$, the quantity
\begin{equation}
    \text{sd}_{\lambda}(A)= \sum_{i=1}^k\frac{1}{1 + \frac{\lambda}{\sigma_i(A)^2}}
\end{equation}
is the statistical dimension. 
\end{definition}
This has received significant attention in the randomized NLA literature  as recent work \cite{Cohen2016, Avron2017, Avron2017a} shows the sketch size can be of the same order as the statistical dimension, which is bounded above by the rank of the matrix. Thus when $\text{sd}_{\lambda}(A)\ll \min(m,n)$, 
we are able to use a sketch dimension smaller than $\min(m,n)$.  The algorithm we introduce has as two main advantages compared to previous work that 1) it is not necessary to compute the Gram matrix (which could lead to amplified numerical instability), and 2) it requires a decomposition of size proportional to the sketching dimension instead of $m$ or $n$. 

In the next section we introduce both of these algorithms generally and in Section \ref{sec:relatedwork} we discuss related work. 

\subsection{Our contribution}

A naive randomized approach to solving \eqref{eq:regLLS} for multiple values of $\lambda$ would be to solve the problem from scratch for each $\lambda_i$. Our approach is to sketch $A$ only once and reuse this sketch for each of the $\lambda_i$.  Reusing sketches in Tikhonov regularization was similarly suggested in~\cite{Meng2014a} for sketch-to-precondition, and follows naturally in various works~\cite{Wang2018,Chen2015FastRegression,Chowdhury2018AnRegression}. Avron et al.  \cite{Avron2017} introduced the name `partly exact' sketching. In our context, a partly exact sketch of the matrix in \eqref{eq:regLLSover}, would be
\begin{equation} \label{eq:sketch1}
\begin{bmatrix} XA \\ \sqrt{\lambda} I_n \end{bmatrix},
\end{equation} 
for an embedding matrix $X\in\R^{s\times m}$, $s\ll m$ (as opposed to $XB$). An embedding matrix is a random matrix that `embeds' a matrix in a lower dimension while preserving as much information as possible. An exact definition is provided in Section \ref{sec:notation}.
 Standard sketch-to-precondition practice is to compute the QR factorization of the sketch \eqref{eq:sketch1} and to then use the $R$ factor as a preconditioner for \eqref{eq:regLLSover}.  The Randomized NLA literature tells us that this would lead to a good preconditioner for an appropriate type of embedding and sketch dimension. 

Our work builds heavily on these concepts. Consider first the case that the statistical dimension is not much smaller than the statistical dimension, so the necessary sketch size would likely be larger than $\min(m,n)$.  We propose to compute the Cholesky factorization of the Gram matrix of the sketch, $R^TR=(XA)^T(XA) + \lambda I$, which we know is possible since the use of regularization makes $B$ numerically full rank, so $R$ can be computed without breakdown~\cite[Ch.~10]{Higham2002}. We show that with high probability the Cholesky factor $R$ is such that $BR^{-1}$, where $B$ is as in \eqref{eq:regLLSover}, is well conditioned. As a result,  if $R$ is used as a preconditioner, LSQR~\cite{Paige1982} converges geometrically.

This approach allows for an easy update for multiple regularization parameters: the preprocessing steps consist of sketching $A$ to obtain $Y = XA$ and computing the Gram matrix $C = \YY$. Say $X\in\R^{s\times m}$, where $n\leq s \ll m$. We can then bound the operations necessary for the preprocessing steps above by $\mathcal{O}(mn\log(m) + sn^2)$.  For each $\lambda_i$ we will consequently only have the cost to compute the Cholesky factor of an $n\times n$ matrix, which has complexity $\mathcal{O}(n^3)$ with a small constant.  Especially for a large number of values of $\lambda$ and/or for a sketching dimension $s$ considerably larger than $n$ (which might be necessary for, for instance, spare embeddings), this can outperform computing the QR factorization of an $s\times n$ matrix in terms of cost. The same argument can also be applied to the underdetermined case; both algorithms are presented in detail in Section 2.

Next consider the situation where the statistical dimension is orders of magnitude smaller than $\min(m,n)$.  Recent work~\cite{Cohen2016, Avron2017, Chowdhury2018AnRegression} shows that one can use a sketch dimension $s$ proportional to the statistical dimension $\text{sd}_{\lambda}(A)$, instead of proportional to $\min(m,n)$, and obtain useful sketches.  Inspired by the Kernel Ridge Regression (KRR) solver~\cite{Avron2017a}, we propose a scheme that results in a preconditioner based on this small sketch.  The main contribution of this work is an algorithm which requires only $\mathcal{O}(T_{\text{sketch}} + \min(m,n)s^2)$ operations to find a preconditioner that is suitable for any value of $\lambda$, where $T_{\text{sketch}}$ is the operations required to sketch $A$.  We are able to do so by computing the SVD of the small sketch $XA$, and consequently using the Woodbury matrix identity and the (truncated) SVD factors of $XA$ to compute a preconditioner with low-rank structure that can be applied cheaply. The algorithm is applicable in both the underdetermined and the overdetermined case.  

Although our work is close in spirit to~\cite{Avron2017a} (translated to an LLS context) and~\cite{Chowdhury2018AnRegression}, these works are concerned with solving the normal equations, thus involving the Gram matrix $\AA$. This results in squaring the condition number (as $\lambda\rightarrow 0$). It is well-known for LLS that solution via the normal equation is unstable, whereas backward stability can be recovered if one works directly with the matrix $A$~\cite[Ch.~20]{Higham2002}. Our approach does this, and finds a preconditioner for $A$ rather than $\AA$. Furthermore, we avoid $\mathcal{O}(\min(m,n)^3)$ work and instead only perform $\mathcal{O}(\min(m,n)\text{sd}_{\lambda}(A)^2)$ operations.

As far as the authors are aware, this is the only sketch-to-precondition algorithm that efficiently makes use of a sketch dimension smaller than $\min(m,n)$ while avoiding the normal equations.  Additionally, we propose a method to estimate $\text{sd}_{\lambda_i}(A)$ for each $\lambda_i$ within our algorithms without any additional cost, based on \cite{Meier2021}.

We finally present convergence results for both algorithms based on the structural conditions presented in Chowdhury et al. (2018) \cite{Chowdhury2018AnRegression}.

In conclusion, the main contributions of this paper are
\begin{itemize}
	\item Computational choices and analysis focused on a context where the optimal regularization parameter is unknown, and we wish to solve \eqref{eq:regLLS} for multiple values of $\lambda$.
	\item A broad treatment that covers underdetermined ($n\gg m$) and overdetermined ($m\gg n$) cases and distinguishes between the cases $\min(m,n) = \mathcal{O}(\text{sd}_{\lambda}(A))$ and $\min(m,n) \gg \text{sd}_{\lambda}(A)$.
	\item The design of a novel algorithm for preconditioning \eqref{eq:regLLS} for problems with statistical dimension much smaller than $\min(m,n)$ based on the Woodbury matrix identity.  It requires one decomposition and can then handle multiple values of $\lambda$. Furthermore, as opposed to previous work, the algorithm finds a preconditioner for $A$ directly instead of for the normal equations. 
\end{itemize}

\subsection{Related work}\label{sec:relatedwork}
The regularized system \eqref{eq:regLLS} can be solved with direct methods based on the QR decomposition in $\mathcal{O}(mn\min(m,n))$ operations. For very large scale systems, this cost is prohibitively large and we must turn to iterative solvers and/or randomized methods.  Krylov subspace-based iterative solvers such as LSQR (which we suggest by default), conjugate gradients and Chebyshev semi-iterative techniques require $\mathcal{O}(mn)$ work per iteration, which could be much less than direct methods if the number of iterations necessary for convergence is $\ll \min(m,n)$. However, the number of iterations needed grows as the condition number of the matrix grows, which can make iterative solvers slow. One solution is to use a preconditioner, which is discussed below. 

There is a wide variety of randomized techniques that can be subdivided in two main categories. First there are sketch-and-solve algorithms which replace the regularized system \eqref{eq:regLLSover} with
\begin{equation}
	\min_{x\in\R^n}\|X(Ax - b)\|^2_2 + \lambda \|x\|_2^2,
\end{equation}
where $X\in\R^{s\times m}$, $s\ll m$, is a random embedding matrix \cite{Drineas2011}. In \cite{Avron2017} it is shown that $s$ can be chosen to be $\tilde{\mathcal{O}}(\text{sd}_{\lambda}(A)/\epsilon)$ to obtain a solution within ($1+\epsilon$) relative residual.  The smaller sketched system may be solved iteratively, for instance with the normal equations or in the dual space for underdetermined problems~\cite{Lu2013FasterTransform}.  Chen et al. \cite{Chen2015FastRegression} propose an algorithm to sketch-and-solve the normal equations efficiently.

Secondly, we distinguish a class of sketch-to-precondition algorithms, which were introduced in Section \ref{sec:intro}.  Here, the system is solved with a deterministic iterative algorithm such as LSQR using a preconditioner that was obtained through sketching. Our algorithms fall into this category. Important classical works include \cite{Rokhlin2008}, \cite{Meng2014a} and \cite{Avron2010a}. Recent high-performance implementations can be found in \cite{Iyer2015ASystems} and \cite{Iyer2017RandomizedProblems}.  This type of algorithm has an $\mathcal{O}(\log(1/\epsilon))$ dependency on the accuracy. The algorithms in these references all require the decomposition of an $s\times \min(m,n)$ matrix, which can be costly.  In this work we propose algorithms that require smaller decompositions.

Recently, a new subclass of sketch-to-precondition algorithms that does not require any decomposition was introduced.  These are iterative methods where sketching is used in each iteration, most notably Iterative Hessian Sketching (IHS) proposed in \cite{Pilanci2016} uses a sketched Hessian and inexact inversion.  Extensions
of this work include Accelarated IHS~\cite{Wang2017a}, Momemtum-IHS \cite{Ozaslan2020, Ozaslan2019}, and Polyak-IHS \cite{Lacotte2020}. Wang et al. discuss more general iterative sketching methods for regularized methods (that might require decompositions) in \cite{Wang2018}, in particular from a statistical perspective. Of works in this category, \cite{Chowdhury2018AnRegression} is closest in spirit to our work. The authors design an iterative algorithm for the underdetermined problem which only requires a sketch (based on leverage scores) of size proportional to the statistical dimension.  A difference is that it does require a $\min(m,n) \times \min(m,n)$ exact inversion.

There are also works in Kernel Ridge Regression (KRR) that are closely related, most notably \cite{Avron2017a} which we discussed above. Chen et al.  \cite{Chen2021AccumulationsRegression} discuss a general framework for sketch-and-solve in KRR and \cite{Avron2017} analyses errors in this context. In \cite{ElAlaoui2020}, the authors propose a sketch-and-solve algorithm based on leverage score sampling and in  \cite{Yang2017RandomizedRegression} a sketch-and-solve algorithm based on a different definition of the statistical dimension is introduced. An iterative sketching algorithm is proposed in \cite{Gazagnadou2021}.

Our paper aims to contribute to the literature in two main ways: firstly by explicitly considering the situation of multiple regularization parameters and making computational choices especially fit for that purpose, and secondly by proposing a preconditioner for the case $\text{sd}_{\lambda}(A) \ll \min(m,n)$ that is computationally efficient for multiple $\lambda$ and avoids the normal equations. Furthermore, our analysis is broad as it allows for the underdetermined and overdetermined case. 

\subsection{Notation}\label{sec:notation}
Throughout this paper $A\in\R^{m\times n}$ is a real matrix that has either $m\gg n$ or $n\gg m$.  A double subscript under a matrix denotes an entry, for instance, $(A)_{ij}$ denotes the $(i,j)$th entry of $A$. For vectors, $\|a\|_2$ denotes the Eucledian norm; for matrices $\|A\|_2 = \|A\|$ denotes the spectral norm and $\|A\|_F$ the Frobenius norm. We let $X$ denote an embedding matrix. That is, for $X\in\R^{s\times m}$, $s\ll m$, $X$ is a random matrix such that $\|U^TX^TXU - I\|_2 \leq \epsilon$ with high probability for matrices $U$ with orthonormal columns and $0 < \epsilon < 1$. In general, $s$ denotes the sketch dimension.

\section{Randomized Tikhonov regularization with Cholesky}

We start with the exposition of a scheme that fits the classical regime; $s\geq\min(m,n)$, where $s$ is the dimension of the sketch. This would be the setting when the singular values of $A$ are not decaying rapidly, and so $\min(m,n)$ is close to $\text{sd}_{\lambda}(A)$.
The algorithm we propose varies from most of the existing literature in two main ways. Firstly, we use a `partly exact' sketch to improve the efficiency of computing for multiple $\lambda$, and secondly, we propose a Cholesky-decomposition based computation as opposed to a QR decomposition based computation.  The latter is possible since the matrices $B$ from \eqref{eq:regLLSover} and $D$ from \eqref{eq:regLLSunder} are not very ill-conditioned because of the regularization. 

In this section we propose two algorithms, one for the overdetermined case and one for the underdetermined case. We treat these cases seperately and discuss the problem setting, the algorithm, and the convergence analysis in each case.
\subsection{Overdetermined case ($m \gg n$)}
Let us start with specifying the problem.
We aim to find preconditioners for the system \eqref{eq:regLLSover} for multiple regularization parameters $\lambda_1,\dots,\lambda_N$. In particular, by denoting $$B_i = \begin{bmatrix} A \\ \sqrt{\lambda_i} I_n\end{bmatrix}, \quad i = 1,\dots, N,$$
we find preconditioners $R_1, \dots, R_N$ such that $\kappa(B_iR_i^{-1}) = \mathcal{O}(1)$.  Having obtained such an $R_i$,  the solution to \eqref{eq:regLLSover} is obtained with an iterative solver. As a default we use LSQR. The system that is solved iteratively is
\begin{equation}\label{eq:regLLSoverIter}
    y^* = \argmin_{y\in\R^{n}} \|B_iR_i^{-1}y - \hat{b}\|^2_2, \quad x^*_{\lambda_i} = R_i^{-1}y.
\end{equation}
Here, $\hat{b} = [b^T,\, 0]^T\in\R^{m+n}$.

\subsubsection{Algorithm}\label{sec:choloveralg}
The algorithm to compute $R_i$ for $i=1,\dots, N$ is straightforward, as mentioned in Section 1.1. We sketch the matrix $A$ once from the left to obtain a matrix $Y=XA\in\R^{s\times n}$. Usually, in the context we are considering, $s\geq n$, but the algorithm could be applied in the case $s<n$.\footnote{However, we would recommend using the algorithms discussed in Section 3 in this case.} The next preprocessing step is to compute the Gram matrix of $Y$: $C = \YY$.  Note that if $A$ is ill-conditioned, $C$ will also be ill-conditioned: as sketching $A$ roughly preserves the largest and smallest singular values, the condition number is also roughly preserved.

The work involved in these steps can be bounded above by $\mathcal{O}(mn\log(m) + sn^2)$, where the $\mathcal{O}(mn\log(m))$ operations corresponds to sketching $A$ with a subsampled randomized trigonometric transform (SRTT) embedding matrix~\cite{Martinsson2020}.  See below for a discussion on various embedding matrices.

Having obtained $C$,  we loop through the values of $\lambda$ we wish to compute $x_{\lambda}^*$ for. For each $\lambda_i$, the conditioner $R_i$ is obtained as the Cholesky factor of the positive definite matrix $C + \lambda_i I$. Note this is the 'partly exact' sketched Gram matrix of $B$ or $D$. As is well-known, the Cholesky decomposition breaks down for large condition number. However, because we are considering a regularized problem we make the mild assumption $\kappa_2(C + \lambda_i I)<\mathcal{O}(u^{-1})$. The cost of computing the Cholesky decomposition is $\mathcal{O}(n^3)$, more precisely $\frac{1}{3}n^3$ flops~\cite{Golub2013}. The resulting algorithm is shown in Algorithm \ref{alg:rankCholTikhonovover}.

\begin{algorithm2e}
\SetAlgoLined
\KwResult{Given an $m\times n$ matrix $A$, $m> n$, a set of regularization parameters $\lambda_1,\lambda_2,\dots, \lambda_N$,and a sampling parameter $s<m$, this algorithm computes approximate solutions $x_{\lambda_i}^*$ to \eqref{eq:regLLSoverIter}.}
\nl Draw an $s\times m$ random embedding matrix $X$. \\
\nl Compute $Y = XA$. \\
\nl Compute $C = Y^TY$. \\
 \For{$i = 1,2,\dots, N$} {
\nl Compute $R_i = \text{chol}\,(C + \lambda_i I)$. \\
\nl Solve the following system with LSQR
 \begin{equation}\label{eq:lsqrsolve_1}
     y_i^* = \argmin_{y\in\R^n}\left\|\begin{bmatrix} A\\ \sqrt{\lambda_i}I_n\end{bmatrix}R_i^{-1}y - \begin{bmatrix} b\\ 0\end{bmatrix}\right\|^2.
 \end{equation}
 \nl Set $x_{\lambda_i}^* = R_i^{-1}y_i^*$
}
\caption{Randomized Cholesky for Tikhonov regularization in overdetermined LLS.}
\label{alg:rankCholTikhonovover}
\end{algorithm2e}

The full complexity of the algorithm is $\mathcal{O}(mn\log m +sn^2 + N(n^3 + mn\log(1/\varepsilon)))$ to solve \eqref{eq:regLLSoverIter} for $i=1,\dots, N$ to relative accuracy $\varepsilon$.  We note that solving the $N$ problems in the for loop can trivially be parallelized.

There are different choices for the embedding matrix possible. An SRTT matrix is a specific type of embedding that can be applied quickly. In particular, if $X\in\R^{s\times m}$, $s\ll m$, is an SRTT then it has the form $$X = \sqrt{\frac{m}{s}}SFD,$$
where $S\in\R^{s\times m}$ is a subsampling matrix --- its rows are a random subset of the rows of the $m\times m$ identity matrix ---, $F\in\R^{m\times m}$ is an orthogonal trigonometric transform, such as a discrete cosine or Hadamard transform, and $D\in\R^{m\times m}$ is a diagonal matrix of independent random signs.  SRTTs can be applied in $\mathcal{O}(mn\log(m))$ operations. Other types of embedding matrices include Gaussian matrices, with has each element an independent normal random variable, or a sparse embedding which has one nonzero element in each column. These matrices can be applied in respectively $\mathcal{O}(mns)$ and $\mathcal{O}(\text{nnz}(A))$ operations. The quality of an embedding dictates how large the sketching dimension must be for sufficiently accurate results, see~\cite{Halko2011,  Martinsson2020,  Woodruff2014} for discussions.

\subsubsection{Convergence analysis}
Throughout this paper, we propose convergence proofs based on the structural conditions presented in \cite{Chowdhury2018AnRegression}. These conditions are on the quality of the embedding matrix, in particular how well the relevant information in $A$ is preserved.  For the case $s\geq n$, the relevant condition is $$ \|U_A^TX^TXU_A-I_n\|_2 \leq \tilde{\epsilon},$$
for some $0<\tilde{\epsilon}<1$. Here, $U_A$ is the matrix consisting of the left singular vectors of $A$. We use the closely related 
condition on the singular values of $XU_A$:
$$1 - \epsilon \leq \sigma_{\min}(XU_A)\leq\sigma_{\max}(XU_A) \leq 1 + \epsilon.$$
It results in the following lemma on the condition number of $BR^{-1}$. It is important to note that Algorithm \ref{alg:rankCholTikhonovover} outputs preconditioners $R_i$ that satisfy (in exact arithmetic)
$$R_i^TR_i = A^TX^TXA + \lambda_i I_n.$$

\begin{lemma}\label{lemma:convcholover}
Let $A\in\R^{m\times n}$, $m > n$ have SVD $A = U_A\Sigma_A V_A^T$, where $U_A\in\R^{m\times n}$, and let $\lambda > 0$.
Let $X\in\R^{s\times m}$, $m > s \geq n$, be a matrix such that $1-\epsilon < \sigma_{\min}(XU_A)<\sigma_{\max}(XU_A) < 1 + \epsilon$ for some $0<\epsilon <1$. Suppose $R\in\R^{n\times n}$ is an upper triangular matrix such that
$$ \RR = A^TX^TXA + \lambda I_n.$$ Then
$$\kappa_2(BR^{-1}) \leq \frac{1 + \epsilon}{1-\epsilon},$$
where
\begin{equation}  \label{eq:defB}
B = \begin{bmatrix} A\\ \sqrt{\lambda}I_n\end{bmatrix}.  
\end{equation}
\end{lemma}
\begin{proof}
The proof is inspired by that of \cite[Lem.~15]{Chowdhury2018AnRegression}. We write $S\preceq T$ for Hermitian matrices $S$ and $T$ imply that $S-T$ is negative semi-definite . We start from the fact that $\sigma_i(XU_A)^2 = \lambda_i(U_A^TX^TXU_A)$, where $\lambda_i$ denotes the $i$th greatest eigenvalue. Then
$$ \sigma_{\min}(XU_A)^2 I \preceq U_A^TX^TXU_A \preceq\sigma_{\max}(XU_A)^2  I.$$
Multiply this expression by $V_A\Sigma_A$ from the left and $(V_A\Sigma_A)^T$ from the right to find
$$ \sigma_{\min}(XU_A)^2\AA \preceq A^TX^TXA \preceq\sigma_{\max}(XU_A)^2  \AA,$$
which is equivalent to
$$ \sigma_{\min}(XU_A)^2\AA + \lambda I\preceq A^TX^TXA + \lambda I \preceq\sigma_{\max}(XU_A)^2  \AA+ \lambda I.$$
Under the assumption that $\sigma_{\min}(XU_A)<1<\sigma_{\max}(XU_A)$ this implies
$$ \sigma_{\min}(XU_A)^2(\AA + \lambda I)\preceq A^TX^TXA + \lambda I \preceq\sigma_{\max}(XU_A)^2  (\AA+ \lambda I).$$
We can rephrase this as
$$ \sigma_{\min}(XU_A)^2\BB\preceq \RR \preceq\sigma_{\max}(XU_A)^2  \BB.$$
It follows that
$$\sigma_{\max}(XU_A)^{-2}I\preceq(BR^{-1})^T(BR^{-1})  \preceq\sigma_{\min}(XU_A)^{-2} I.$$
This finally implies
$$\sigma_{\max}(XU_A)^{-1}\leq\sigma_{\min}(BR^{-1})\leq \sigma_{\max}(BR^{-1})\leq \sigma_{\min}(XU_A)^{-1},$$
so we have
$$\kappa_2(BR^{-1}) \leq \kappa_2(XU_A),$$
as required.
\end{proof}
It is well-known that LSQR applied to a well-conditioned LLS converges geometrically with respect to the number of iterations. It follows from the above result that we obtain an $\varepsilon$-accuracy solution for~\eqref{eq:lsqrsolve_1} with $\mathcal{O}(\log(1/\varepsilon)$ iterations of LSQR. Note that $\varepsilon$ is different from $\epsilon$ in the condition of Lemma \ref{lemma:convcholover}. In practise, $\epsilon$ is usally not much smaller than 1, say 0.5, whereas $\varepsilon$ can be much smaller.

\subsection{Underdetermined case ($n\gg m)$}
We next turn to the underdetermined case \eqref{eq:regLLSunder}. Let us first summarise the problem set-up: we aim to find preconditioners $R_i$ for $$D_i = \begin{bmatrix} A & \sqrt{\lambda_i} I_m\end{bmatrix}, \quad i = 1,\dots, N$$
such that $\kappa(R_i^{-T}D_i) = \mathcal{O}(1)$. We then use an iterative solver to find the minimum-norm solution to
\begin{equation}\label{eq:regLLSunderiter}
    \begin{bmatrix} x^*_{\lambda_i} \\ y^* \end{bmatrix} = \argmin_{x\in\R^n, y\in\R^{m}} \left\|R_i^{-T}D_i\begin{bmatrix} x \\ y \end{bmatrix} - R_i^{-T}b\right\|^2_2.
\end{equation}
The first $n$ coordinates of the minimum-norm minimizer of \eqref{eq:regLLSunderiter}, $x^*_{\lambda_i}$, is the approximate solution.
\subsubsection{Algorithm}
The idea for the algorithm is very similar to Section \ref{sec:choloveralg}. The main difference is that now we sketch from the right to obtain $Y = AX\in\R^{m\times s}$, for $s\ll n$ and work with the Gram matrix $C = YY^T\in\R^{m\times m}$. The preconditioners $R_i$ are the Cholesky factors such that
$$ R_i^TR_i = C + \lambda_i I_m = (AX)(AX)^T + \lambda_i I_m.$$
The cost of preprocessing can be bounded above by $\mathcal{O}(mn\log(n) + sm^2)$ operations. The Cholesky decomposition has $\mathcal{O}(m^3)$ complexity. The resulting algorithm is presented in \ref{alg:rankCholTikhonovunder}.

\begin{algorithm2e}
\SetAlgoLined
\KwResult{Given an $m\times n$ matrix $A$, $n > m,$, a set of regularization parameters $\lambda_1,\lambda_2,\dots, \lambda_N$,and a sampling parameter $s<n$, this algorithm computes approximate solutions $x_{\lambda_i}^*$ to \eqref{eq:regLLSunderiter}.}
\nl Draw $n\times s$ random embedding matrix $X$. \\
\nl Compute $Y = AX$. \\
\nl Compute $C = YY^T$. \\
 \For{$i = 1,2,\dots, N$} {
\nl Compute $R_i = \text{chol}\,(C + \lambda_i I_m)$. \\
\nl Compute the min-length solution to the following system with an iterative solver
 \begin{equation*}
    \begin{bmatrix} x^*_{\lambda_i} \\ y^*_i \end{bmatrix} = \argmin_{x\in\R^n, y\in\R^{m}} \|R_i^{-T}\begin{bmatrix} A & \sqrt{\lambda_i}I \end{bmatrix}\begin{bmatrix} x \\ y \end{bmatrix} - R_i^{-T}b\|^2_2.
\end{equation*} \\
\nl Return $x^*_{\lambda_i}$.
}
\caption{Randomized Cholesky for Tikhonov regularization in underdetermined LLS.}
\label{alg:rankCholTikhonovunder}
\end{algorithm2e}

The complexity of the algorithm is $\mathcal{O}(mn\log n + sm^2 + N(m^3 + mn\log(1/\varepsilon)))$ to obtain relative accuracy $\varepsilon$. 

\subsubsection{Convergence analysis}
We again present a result on the condition number of the preconditioned matrix in terms of the quality of the sketching matrix. Specifically,  we consider the singular values of $X^TV_A$ where $V_A$ is the matrix consisting of the right singular vectors of $A$.
\begin{lemma}\label{lemma:convcholunder}
Let $A\in\R^{m\times n}$, $n > m$ have SVD $A = U_A\Sigma V_A^T$, where $V_A\in\R^{n\times m}$ and let $\lambda > 0$.
Let $X\in\R^{n\times s}$, $n > s > m$, be an embedding matrix such that $1 - \epsilon < \sigma_{\min}(X^TV_A)\leq\sigma_{\max}(X^TV_A) < 1 + \epsilon$ for some $0<\epsilon <1$.  Suppose $R\in\R^{n\times n}$ is an upper triangular matrix such that
$$ \RR = AXX^TA^T + \lambda I_m.$$ Then
$$\kappa_2(R^{-T}D) \leq \frac{1 + \epsilon}{1-\epsilon},$$
where 
\begin{equation}  \label{eq:defD}
D = \begin{bmatrix} A& \sqrt{\lambda}I_m\end{bmatrix}.  
\end{equation}
\end{lemma}
\begin{proof}
The proof of Lemma \ref{lemma:convcholunder} follows the exact same steps as the proof of Lemma \ref{lemma:convcholover}, and so is omitted. 
\end{proof}

\section{Randomized Tikhonov regularization for problems with small statistical dimension}

A regularized linear least squares problem with an approximately low-rank data matrix, i.e. $A$ has rapidly decreasing singular values, may be solved in a way that takes advantage of this particular structure. Specifically, we may sketch the matrix with an embedding matrix that is smaller than either of the original dimensions. In this section we discuss an algorithm that computes a preconditioner for this context.  

To see why one can be more efficient here, consider the Gram matrix of the sketched data matrix. Say $X$ is an embedding matrix of dimension $s<\min(m,n)$; then $(XA)^T(XA)$ in the overdetermined case or $(AX)(AX)^T$ in the underdetermined case are now exactly low rank matrices. We would prefer to avoid the decomposition of an $\min(m,n)$ by $\min(m,n)$ matrix, as would be necessary in Algorithms 1 and 2, and instead work with the smaller sketch directly.

We take inspiration from Avron et al (2017)~\cite{Avron2017a}, which discusses a sketch-to-precondition algorithm for kernel ridge regression.  The authors use the Woodbury matrix identity to (in our notation for the overdetermined system) use the Cholesky decomposition of $(XA)(XA)^T + \lambda I_s$ to obtain a preconditioner for the normal equations mathematically equivalent to $P = (XA)^T(XA) + \lambda I_n$.  That is, the matrix they are interested in is $P$ whereas we are interested in obtaining a matrix $R$ such that $P = \RR$. This is important to avoid instabilities caused by solving the normal equations (and thus computing the Gram matrix of $A$), as opposed to preconditioning $B$ or $D$ (in~\eqref{eq:defB},\eqref{eq:defD}) directly and solving the LLS problem. This is a notable difference between our work and earlier work, and our experiments illustrate the improved stability.

In this section we assume $A$ to be approximately low rank, with decaying singular values such that $\text{sd}_{\lambda}(A)$ as defined in Definition \ref{def:statdim}, is (much) smaller than $\min(m,n)$. Again, we consider a sequence of regularization parameters $\lambda_1\geq\dots\geq\lambda_N$ for which we aim to solve \eqref{eq:LLs} with an iterative solver.  We propose an algorithm that finds a sequence of preconditioners for either $B_i$ and $D_i$, where
\begin{equation*}
B_i  = \begin{bmatrix}A \\ \sqrt{\lambda_i}I_n\end{bmatrix}, \quad D_i = \begin{bmatrix}A & \sqrt{\lambda_i}I_m\end{bmatrix},  \quad i = 1,\dots,N,
\end{equation*}
in $\mathcal{O}(mn\log(\max(m,n)) + \min(m,n)\text{sd}_{\lambda}(A)^2)$ operations. The preconditioners can be applied to a vector in $\mathcal{O}(\min(m,n)\text{sd}_{\lambda}(A))$ operations.  We first derive the preconditioner in the overdetermined case.

\subsection{Overdetermined case}
We suppose in this subsection that $A\in\R^{m\times n}$ where $m\gg n$, and we fix a $\lambda$ (for the sake of the argument) such that $\text{sd}_{\lambda}(A)\ll n$. We look for a preconditioner for $ B =[ A^T \,\, \sqrt{\lambda}I_n]^T$.
Let $X\in\R^{s\times m}$ be an embedding with $s\ll n\ll m$ and $s = \mathcal{O}(\text{sd}_{\lambda}(A))$.  We use the (short and fat) sketch
\begin{equation*}
    Y = XA \in\R^{s\times n}.
\end{equation*}
Our aim is to find $R$ such that 
\begin{equation}\label{eq:Rlowrankeq}
    \RR = Y^TY + \lambda I_n,
\end{equation}
as we will show this results in a good preconditioner for $B$.  In the previous section we suggested a triangular matrix; now we propose a preconditioner with a low rank structure.
First take the SVD  of $Y$:
\begin{equation*}
    Y = U\Sigma V^T,\quad U\in\R^{s\times s}, \quad \Sigma \in \R^{s\times s}, \quad V\in\R^{n\times s}.
\end{equation*}
Then
\begin{align*}
    Y^TY + \lambda I_n &= V\Sigma^2V^T + \lambda I_n \\
    & = \begin{bmatrix} V & V_{\perp}\end{bmatrix}\begin{bmatrix}
    \Sigma^2 + \lambda I_s & 0 \\ 0 & \lambda I_{n-s}\end{bmatrix}\begin{bmatrix} V^T \\ V_{\perp}^T\end{bmatrix},
\end{align*}
where $V_{\perp}\in\R^{n\times (n-s)}$ is such that $\begin{bmatrix} V & V_{\perp}\end{bmatrix}$ is a square orthogonal matrix. If we let 
\begin{equation}
    R = \begin{bmatrix} V & V_{\perp}\end{bmatrix}\begin{bmatrix}
    (\Sigma^2 + \lambda I_n)^{1/2} & 0 \\ 0 & \sqrt{\lambda} I_{n-s}\end{bmatrix}\begin{bmatrix} V^T \\ V_{\perp}^T\end{bmatrix},
\end{equation}
then $\RR = Y^TY + \lambda I_n$. Note we can reformulate $R$ to not include $V_{\perp}$ as follows
\begin{align*}
    R &= V(\Sigma^2 + \lambda I_n)^{1/2}V^T + \sqrt{\lambda}V_{\perp}V_{\perp}^T \\
    & = V(\Sigma^2 + \lambda I_n)^{1/2}V^T + \sqrt{\lambda}(I_{n} - VV^T) \\
    & =  V[(\Sigma^2 + \lambda I_n)^{1/2} - \sqrt{\lambda}I_s]V^T + \sqrt{\lambda}I_n \\
    & = \sqrt{\lambda}(VFV^T + I_n),
\end{align*}
where 
\begin{equation*}
    F = (\lambda^{-1/2}\Sigma^2 + I_s)^{1/2} - I_s
\end{equation*}
is a diagonal matrix. We are interested in the inverse of $R$. By using the Sherman-Morrison-Woodbury formula we have
\begin{align*}
    (VFV^T + I_n)^{-1} &= I_n - V(F^{-1} + V^TV)^{-1}V^T \\
    & = I_n - VSV^T,
\end{align*}
where $S = (F^{-1} + I_n)^{-1}$ is a diagonal matrix with elements
\begin{align*}
    S_{ii} = 1 - \sqrt{\frac{1}{1 + \frac{\sigma_i(Y)^2}{\lambda}}}\,\,.
\end{align*}
Finally,
\begin{equation}\label{eq:RinvOver}
    R^{-1} = \lambda^{-1/2}(I_n - VSV^T),
\end{equation}
where $V\in\R^{n\times s}$ and $S\in\R^{s\times s}$.  This matrix, although not triangular, has a low-rank structure that allows it to be applied quickly. Most importantly, it is such that \eqref{eq:Rlowrankeq} holds while we avoided the decomposition of an $n\times n$ matrix.

\subsubsection{The algorithm}
In the previous section we focused on the context where we aim to solve \eqref{eq:regLLS} for a sequence of regularization parameters $\lambda_1\geq\dots\geq\lambda_N$. The way Algorithms \ref{alg:rankCholTikhonovover} and \ref{alg:rankCholTikhonovunder} were designed allowed us to sketch the matrix $A$ only once, and reuse the sketch for different values of $\lambda_i$.  For problems with low statistical dimension it is slightly more difficult to do this, because we want our sketch size to depend on $\text{sd}_{\lambda_i}(A)$ which in turn depends on $\lambda_i$.  We propose an algorithm to bypass this problem.

As the statistical dimension increases as $\lambda$ decreases, the largest statistical dimension corresponds to the smallest $\lambda$ --- which we assume $\lambda_N$. Our algorithm sketches $A$ once, with dimension corresponding to $\text{sd}_{\lambda_N}(A)$ and compute the SVD of this sketch.  For all other values of $\lambda$, we truncate the factors of the SVD and so use lower-rank matrices that correspond to $\text{sd}_{\lambda_i}(A)$.

The resulting algorithm is presented in Algorithm \ref{alg:randLowRankTikhonov}. We are able to find a preconditioner in $\mathcal{O}(mn\log m + ms^2)$ operations. It can be applied to a vector $x\in\R^{n}$ by computing
\begin{equation*}
    R^{-1}x = \lambda^{-1/2}(x - VSVx),
\end{equation*}
which requires $2(2ns + s^2 + n)$ operations --- compared to $n^2$ for an upper triangular preconditioner. One of its main advantages is that we only need to compute a preconditioner once, and for any (sufficiently large) $\lambda$, one can truncate this without any additional computational work.

\begin{algorithm2e}
\SetAlgoLined
\KwResult{Given an $m\times n$ matrix $A$, $m> n$, a set of regularization parameters $\lambda_1\geq\lambda_2\geq \dots\geq\lambda_N$,  estimates of $\text{sd}_{\lambda_i}(A)$ for $i=1,\dots, N$,  and an oversampling parameter $\alpha$ such that $\alpha\text{sd}_{\lambda_N}(A) \leq n$, this scheme computes approximate solutions $x_{\lambda_i}^*$ to \ref{eq:regLLSoverIter}.}
\nl Set $s = \alpha\ceil{\text{sd}_{\lambda_N}(A)}$.\\
\nl Draw an $s\times m$ random embedding matrix $X$. \\
 \nl Compute $Y = XA$. \\
 \nl Compute the SVD $[\,\sim, \,\Sigma,\, V \,] = \text{svd}\,(Y, \text{`econ'})$. \\
  \For{$i = 1,2,\dots, N$} {
  	\nl Set $s_i = \alpha\ceil{\text{sd}_{\lambda_i}(A)}$.\\
  	\nl Truncate $V_i = V(:, 1:s_i)$ and $\Sigma_i = \Sigma(1:s_i, 1:s_i)$.\\
  \nl	Form a diagonal matrix $S_i$ with elements $$(S_i)_{jj} = 1 - \sqrt{\frac{1}{1 + \frac{(\Sigma_i)_{jj}^2}{\lambda}}}.$$\\
 \nl Solve the following system with LSQR
 \begin{equation*}
     y^* = \argmin_{y\in\R^n}\left\|\begin{bmatrix} A\\ \sqrt{\lambda_i}I_n\end{bmatrix}R_i^{-1}y - \begin{bmatrix} b\\ 0\end{bmatrix}\right\|^2,
 \end{equation*}
 where $$R_i^{-1} = \lambda_i^{-1/2}(I_n - V_iS_iV_i^T).$$\\
 \nl Return $x_{\lambda_i}^* = R_i^{-1}y^*$.
 }
\caption{Randomized preconditioning for Tikhonov regularization in overdetermined LLS.}
\label{alg:randLowRankTikhonov}
\end{algorithm2e}

An important algorithmic consideration is the estimation of the statistical dimension for the various values of $\lambda$.  This is discussed in the next section.

\subsubsection{Estimating the statistical dimension}
In Algorithm \ref{alg:randLowRankTikhonov} it is necessary to obtain estimates of the statistical dimension for various values of $\lambda$. One approach is to sketch $A$ with an embedding matrix of size $s$ considered to be an upper bound for the statistical dimension. This could be $\min(m,n)$ if no information is known, as the sketching step is not the dominant part of the algorithm. The singular values of $Y = XA$ can then be used as estimates for the leading singular values of $A$, as described in \cite{Meier2021}.  In particular, for each value of $\lambda_i$ we can find an estimate $\hat{\text{sd}}_{\lambda_i}(A)$ for the statistical dimension with
$$\hat{\text{sd}}_{\lambda_i}(A) = \sum_{i = 1}^s \frac{1}{1 + \frac{\lambda_i}{\sigma_i(Y)^2}}.$$

Although we may need to compute the SVD of a matrix with greater dimensions than necessary \emph{once}, a good estimate of $\text{sd}_{\lambda_i}(A)$ for each value of $\lambda_i$ will allow us to truncate the low-rank structure in the preconditioner, making it very cheap to apply.  Note that these estimates come without any additional cost, as computing the SVD of $Y$ is a necessary step in Algorithm \ref{alg:randLowRankTikhonov}.

Other methods to estimate the statistical dimension are using a randomized trace estimator such as described in \cite{Avron2011}. However, this will involve solving a linear system, possibly approximately, as in \cite{Ozaslan2020}. Avron et al. propose a new method in \cite{Avron2017} with which the statistical dimension can be estimated in $\mathcal{O}(\text{nnz}(A))$ time up to a constant factor, yet as is noted in \cite{Ozaslan2020}, this is exclusively applicable in a context of very rapid decay of $\sigma_i(A)$.

\subsection{Underdetermined case}
We can perform a very similar analysis for the underdetermined case. Now assume $s<m<n$ and our embedding matrix $X$ is $m\times s$. We find a tall and skinny sketch
\begin{equation*}
    Y = AX\in\R^{m\times s},
\end{equation*}
and aim to find $R$ such that 
\begin{equation*}
    \RR = YY^T + \lambda I_m.
\end{equation*}
Suppose we have the SVD of $Y$
\begin{equation}
    Y = U\Sigma V^T, \quad U\in\R^{m\times s}, \quad \Sigma\in\R^{s\times s}, \quad V\in\R^{s\times s},
\end{equation}
then
\begin{align*}
    YY^T + \lambda I_n &= U\Sigma^2U^T + \lambda I_n \\
    & = \begin{bmatrix} U & U_{\perp}\end{bmatrix}\begin{bmatrix}
    \Sigma^2 + \lambda I_s & 0 \\ 0 & \lambda I_{n-s}\end{bmatrix}\begin{bmatrix} U^T \\ U_{\perp}^T\end{bmatrix},
\end{align*}
where $U_{\perp}\in\R^{m\times (m-s)}$ is such that $\begin{bmatrix} U & U_{\perp}\end{bmatrix}$ is a square orthogonal matrix. By the exact same reasoning as in the overdetermined case, we find
\begin{equation*}
    R = \sqrt{\lambda}(UFU^T + I_n),\quad F = (\lambda^{-1/2}\Sigma^2 + I_s)^{1/2} - I_s.
\end{equation*}
As for the inverse, we have
\begin{equation}
    R^{-T} = \lambda^{-1/2}(I_m - USU^T), \quad \text{where}\quad  S_{ii} = 1 - \sqrt{\frac{1}{1 + \frac{\sigma_i(Y)^2}{\lambda}}}\,\,.
\end{equation}
Here, $U\in\R^{m\times s}$ and $S\in\R^{s\times s}$ is a diagonal matrix with elements as given above. 

The resulting algorithm is presented in Algorithm \ref{alg:randLowRankTikhonovUnder}.

\begin{algorithm2e}
\SetAlgoLined

\KwResult{Given an $m\times n$ matrix $A$, $m< n$, a set of regularization parameters $\lambda_1\geq\lambda_2\geq\dots\geq\lambda_N$, estimates of $\text{sd}_{\lambda_i}(A)$ for $i=1,\dots, N$,  and an oversampling parameter $\alpha$ such that $ \alpha \text{sd}_{\lambda_N}(A) \leq m$, this scheme computes approximate solutions $x_{\lambda_i}^*$ to \eqref{eq:regLLSunderiter}.}
\nl Set $s = \alpha\ceil{\text{sd}_{\lambda_N}(A)}$.\\
\nl Draw an $n\times s$ random embedding matrix $X$. \\
 \nl Compute $Y = AX$. \\
 \nl Compute the SVD $[U, \,\Sigma,\, \sim\,] = \text{svd}\,(Y, \text{`econ'})$. \\
  \For{$i = 1,2,\dots, N$} {
  	\nl Set $s_i = \alpha\ceil{\text{sd}_{\lambda_i}(A)}$.\\
  	\nl Truncate $U_i = U(:, 1:s_i)$ and $\Sigma_i = \Sigma(1:s_i, 1:s_i)$.\\
  \nl	Form a diagonal matrix $S_i$ with elements $$(S_i)_{jj} = 1 - \sqrt{\frac{1}{1 + \frac{(\Sigma_i)_{jj}^2}{\lambda}}}.$$\\
 \nl Solve the following system with LSQR
 \begin{equation*}
    \begin{bmatrix} x^*_{\lambda_i} \\ y^* \end{bmatrix} = \argmin_{x\in\R^n, y\in\R^{m}} \left\|R_i^{-T}\begin{bmatrix} A & \sqrt{\lambda_i}I \end{bmatrix}\begin{bmatrix} x \\ y \end{bmatrix} - R_i^{-T}b\right\|^2_2,
\end{equation*}
 where $$R_i^{-T} = \lambda_i^{-1/2}(I_m - U_iS_iU_i^T).$$\\
 \nl Return $x_{\lambda_i}^*$.
 }
\caption{Randomized preconditioning for Tikhonov regularization in underdetermined LLS.}
\label{alg:randLowRankTikhonovUnder}
\end{algorithm2e}

\subsection{Convergence analysis}
Throughout this analysis, let $A = U_A\Sigma_AV_A^T$ be the economy-sized SVD of $A$. We again employ the structural conditions proposed in \cite{Chowdhury2018AnRegression}.  The proofs on the condition numbers of the preconditioned matrices are again inspired by the proofs in \cite{Chowdhury2018AnRegression}.  We first show these conditions are equivalent to the conditions used in \cite{Avron2017},  as also follows from Lemma 12 in \cite{Avron2017}. 
\subsubsection{Overdetermined case}
The following lemma relates the condition in \cite{Chowdhury2018AnRegression} to the condition in \cite{Avron2017}, and their relation to the statistical dimension.
\begin{lemma}\label{lemma:conditionsequiv}
Let $U_1$ consist of the first $m$ rows of the left singular vectors of $B$ such that
$$ B = \begin{bmatrix} A \\ \sqrt{\lambda}I \end{bmatrix} = \begin{bmatrix} U_1 \\ U_2 \end{bmatrix}\Sigma_BV_B.$$
Define a diagonal matrix $\Sigma_{\lambda}$ by
$$\Sigma_{\lambda} = \Sigma_A(\Sigma_A^2 + \lambda I)^{-1/2} = (\Sigma_A^2 + \lambda I)^{-1/2}\Sigma_A.$$
For an embedding matrix $X\in\R^{s\times m}$, $s<n<m$, we have
\begin{equation*}
    \|U_1^TX^TXU_1 - U_1^TU_1\| = \|\Sigma_{\lambda}U_A^TX^TXU_A\Sigma_{\lambda} - \Sigma_{\lambda}^2\|.
\end{equation*}
Furthermore
\begin{equation*}
    \|U_1\|_F^2 = \|U_A\Sigma_{\lambda}\|_F^2 = \textnormal{sd}_{\lambda}(A).
\end{equation*}
\end{lemma}
\begin{proof}
First note that we have the following for the 
right singular vectors and values of $B$
\begin{equation*}
    \BB = \AA + \lambda I = V_A(\Sigma_A^2 + \lambda I)V_A^T = V_B\Sigma_B^2V_B^T,
\end{equation*}
so $V_B = V_A$ and $\Sigma_B = (\Sigma_A^2 + \lambda I)^{1/2}$.
Then
\begin{equation*}
    \begin{bmatrix} A \\ \sqrt{\lambda}I \end{bmatrix} = \begin{bmatrix} U_1(\Sigma_A^2 + \lambda I)^{1/2}V_A^T \\ U_2(\Sigma_A^2 + \lambda I)^{1/2}V_A^T \end{bmatrix},
\end{equation*}
so that 
\begin{equation*}
    U_1 = AV_A(\Sigma_A^2 + \lambda I)^{-1/2} = U_A\Sigma_A(\Sigma_A^2 + \lambda I)^{-1/2} = U_A\Sigma_{\lambda}.
\end{equation*}
The first result follows; as for the second result, note
\begin{align*}
    \|U_1\|_F^2 = \|\Sigma_{\lambda}\|_F^2 = \|\Sigma_A(\Sigma_A^2 + \lambda I)^{-1/2}\|_F^2 = \sum_{i=1}^n \frac{\sigma^2_i}{\sigma_i^2 + \lambda} = \text{sd}_{\lambda}(A).
\end{align*}
\end{proof}
We use one of these equivalent conditions in the convergence theorems for Algorithms \ref{alg:randLowRankTikhonov} and \ref{alg:randLowRankTikhonovUnder}. 
\begin{theorem}\label{thm:convLRover}
Assume the notation of Lemma \ref{lemma:conditionsequiv} and suppose the following condition holds for some $\epsilon > 0$
\begin{align}
    \|U_1^TX^TXU_1 - U_1^TU_1\| &\leq \epsilon. \label{eq:cond1over}
\end{align}
Let $R$ be such that
\begin{equation*}
    \RR = (XA)^TXA + \lambda I_n.
\end{equation*}
Then
\begin{equation*}
    \kappa(BR^{-1})\leq \sqrt{\frac{1 + \epsilon}{1-\epsilon}}.
\end{equation*}
\end{theorem}
\begin{proof}
We start from condition \eqref{eq:cond1over} to find 
\begin{equation*}
    -\epsilon I \preceq \Sigma_{\lambda}U_A^TX^TXU_A\Sigma_{\lambda} - \Sigma_{\lambda}^2 \preceq \epsilon I.
\end{equation*}
Multiply these inequalities by $V_A$ from the left and $V_A^T$ from the right and use the fact that
$$V_A\Sigma_{\lambda}U_A^T = V_A(\Sigma_A^2 + \lambda I)^{-1/2}V_AV_A^T\Sigma_A = (\AA + \lambda I)^{-1/2}A^T,$$
to obtain
\begin{align*}
    -\epsilon I\preceq (\AA + \lambda I)^{-1/2}A^TX^TXA(\AA + \lambda I)^{-1/2} - \\
    (\AA + \lambda I)^{-1/2}\AA(\AA + \lambda I)^{-1/2} \preceq \epsilon I.
\end{align*}
By multiplying with $(\AA + \lambda I)^{1/2}$ on either side, we find
\begin{equation*}
    -\epsilon \BB\preceq A^TX^TXA - \AA \preceq \epsilon \BB,
\end{equation*}
which is equivalent to
\begin{equation*}
    -\epsilon \BB\preceq A^TX^TXA +\lambda I - (\AA + \lambda I) \preceq \epsilon \BB.
\end{equation*}
We then have
\begin{equation*}
    (1-\epsilon)\BB\preceq \RR \preceq (1+\epsilon) \BB,
\end{equation*}
which results in the following inequalities
\begin{equation*}
    (1+\epsilon)^{-1/2} \leq \sigma_{\min}(BR^{-1}) \leq \sigma_{\max}(BR^{-1}) \leq (1-\epsilon)^{-1/2}.
\end{equation*}
\end{proof}
As mentioned previously, the number of iterations necessary to reach $\varepsilon$ accuracy with LSQR (or conjugate gradient) depends on the condition number of the preconditioned matrix. Specifically, after $\ell$ iterations the ($(BR^{-1})\!^TBR^{-1}$-norm) error is reduced at least by a factor $	\left(\frac{\sqrt{\kappa(BR^{-1})} - 1}{\sqrt{\kappa(BR^{-1}) + 1}}\right)^\ell$~\cite[\S~11.3]{Golub2013}. 
Therefore, with the Cholesky-based preconditioner, Theorem~\ref{thm:convLRover} shows that 
$\ell \geq  \log\left(\frac{(1+\epsilon)^{1/4} + (1-\epsilon)^{1/4}}{(1+\epsilon)^{1/4} - (1-\epsilon)^{1/4}}\right)\log(\frac{1}{\varepsilon})$ iterations suffice for $\varepsilon$-convergence, which is $\mathcal{O}(\frac{1}{\varepsilon})$ if $\epsilon$ is not too large, say $\epsilon<0.8$, as is commonly the case.

\subsubsection{Underdetermined case}
We obtain analogous results for the underdetermined case. 

\begin{lemma}\label{lemma:conditionsequiv2}
Let $A\in\R^{m\times n}$, $m\leq n$, 
and $A= U_A\Sigma_AV_A^T$ be its SVD.
Let $V_1$ consist of the first $m$ rows of the right singular vectors of $D$ such that
$$ D = \begin{bmatrix} A & \sqrt{\lambda}I \end{bmatrix} = U_D\Sigma_D\begin{bmatrix} V_1^T & V_2^T \end{bmatrix}.$$
Define a diagonal matrix $\Sigma_{\lambda}$
$$\Sigma_{\lambda} = \Sigma_A(\Sigma_A^2 + \lambda I)^{-1/2} = (\Sigma_A^2 + \lambda I)^{-1/2}\Sigma_A.$$
For an embedding matrix $X\in\R^{n\times s}$, $s<m<n$, we have
\begin{equation*}
    \|V_1^TXX^TV_1 - V_1^TV_1\| = \|\Sigma_{\lambda}V_A^TXX^TV_A\Sigma_{\lambda} - \Sigma_{\lambda}^2\|.
\end{equation*}
Furthermore,
\begin{equation*}
    \|V_1\|_F^2 = \|V_A\Sigma_{\lambda}\|_F^2 = \text{sd}\,_{\lambda}(A).
\end{equation*}
\end{lemma}
This results in the following theorem on the quality of the preconditioner.

\begin{theorem}
Assume the notation of Lemma \ref{lemma:conditionsequiv2} and suppose the following equivalent condition holds for some $\epsilon > 0$
\begin{align}
    \|V_1^TXX^TV_1 - V_1^TV_1\| &\leq \epsilon \label{eq:cond1under}
\end{align}
Let $R$ be such that
\begin{equation*}
    \RR = AX(AX)^T + \lambda I_m.
\end{equation*}
Then
\begin{equation*}
    \kappa(R^{-T}D)\leq \sqrt{\frac{1 + \epsilon}{1-\epsilon}}\,.
\end{equation*}
\end{theorem}

\subsection{Sketching matrices}
The relationship between condition \eqref{eq:cond1over} and the statistical dimension of the problem can be explained by random matrix multiplication bounds in terms of the stable rank \cite{Cohen2016}. See for instance Corollary 15 and its proof in \cite{Avron2017}. The following theorem describes the size that is required for various oblivious subspace embeddings to obtain a desired level of accuracy.

\begin{theorem}[Ozaslan et al. (2020) \cite{Ozaslan2020}] Let $A\in\R^{m\times n}$, $m\geq n$, and $\lambda>0$ have statistical dimension $\textnormal{sd}_{\lambda}(A)$. Let $U_1$ consist of the first $m$ rows of the left singular vectors of $B = [A^T \,\, \sqrt{\lambda}I_n]^T$ such that
$$ B = \begin{bmatrix} A \\ \sqrt{\lambda}I \end{bmatrix} = \begin{bmatrix} U_1 \\ U_2 \end{bmatrix}\Sigma_BV_B.$$ The condition $$\|U_1^TX^TXU_1 - U_1^TU_1\|\leq \epsilon$$ for an embedding matrix $X\in\R^{s\times m}$ is satisfied with probability at least $1-\delta$ in the following cases:
\begin{itemize}
    \item $X$ is a Sparse Subspace Embedding \cite{Woodruff2014} with one nonzero element in each column and $$s = \Omega(\text{sd}_{\lambda}(A)^2/(\epsilon^2\delta)).$$
    \item $X$ is a Subsampled Randomized Trigonometric Transform \cite{Martinsson2020} and $$s = \Omega((\text{sd}_{\lambda}(A) + \log(1/\epsilon\delta)\log(\text{sd}_{\lambda}(A)/\delta))/\epsilon^2).$$
    \item $X$ is a Sub-Gaussian embedding \cite{Vershynin2012IntroductionMatrices} and $$s = \Omega((\text{sd}_{\lambda}(A)/\epsilon^2)).$$
\end{itemize}
Here, the $\Omega(\cdot)$ notation is defined as $a(n) = \Omega(b(n))$, if there exist two integers $k$ and $n_0$ such that for all $n > n_0$ we have $a(n) \geq k b(n)$.
\end{theorem}

One could also use a subsampling matrix based on leverage scores to satisfy the condition of the theorem, see for instance \cite{Chowdhury2018AnRegression}.

\section{Numerical experiments}
In this section we present numerical experiments to investigate the performance of the algorithms we have introduced. We focus on the overdetermined case ($m\gg n$) and real matrices.  The matrices are constructed as the product of two orthogonal matrices, created as the orthogonal factors of a Gaussian matrix, and a diagonal matrix with the singular values.  Note this results in an incoherent, i.e. easy to sketch, matrix. We set $x\in\R^{n}$ to be a random vector with standard normal entries and then compute $b = Ax + \eta$, where $\eta$ is random noise with approximate norm $\|\eta\|_2 = 10^{-3}$.  

We firstly compare our proposed algorithms to existing methods. Secondly, we investigate the effect of the type of embedding. Finally, in Section \ref{sec:lcurve} we show how our algorithms can be used to find the optimal regularization parameter and resulting solution.

\subsection{Comparison to other methods}\label{sec:expcomp}
We compare our proposed Algorithms \ref{alg:rankCholTikhonovover} and \ref{alg:randLowRankTikhonov} to other randomized LLS solvers.  In particular, we consider two variants of the Blendenpik \cite{Rokhlin2008, Avron2011} algorithm that we label BP1 and BP2 in Figure \ref{fig:comp}. BP1 refers to Blendenpik applied to the full matrix $B$, as suggested in \cite{Iyer2017}. That is, the preconditioner $R_{\text{BP1}}$ is such that 
$$QR_{\text{BP1}} = X\begin{bmatrix}A \\ \sqrt{\lambda} I_n\end{bmatrix}, \quad X\in\R^{s\times (m+n)}.$$
BP2 refers to Blendenpik combined with a `partly exact' sketch. In exact arithmetic this results in the same preconditioner as Algorithm \ref{alg:rankCholTikhonovover}; yet in BP2 it is obtained with a QR decomposition.  The preconditioner $R_{\text{BP2}}$ is such that
$$QR_{\text{BP2}} = \begin{bmatrix}XA \\ \sqrt{\lambda} I_n\end{bmatrix}, \quad X\in\R^{s\times m}.$$
The sketch-to-precondition methods are all combined with LSQR and relative tolerance $10^{-6}$.
Finally, we compare with the Kernel Ridge Regression algorithm in \cite{Avron2010a}, which preconditions the normal equations with a matrix of low-rank structure $P_{KRR}$ given by
$$R_{KRR} = \text{chol}\,(XA(XA)^T + \lambda I_s), \quad U = (XA)^TR_{KRR}^{-1}, \quad P_{KRR}^{-1} = \frac{1}{\lambda}(I - UU^T),$$
where $X\in\R^{s\times m}$ and $s < n$. This preconditioner is applied to solve $$(\AA + \lambda I)x = A^Tb$$ with preconditioned conjugate gradients.

\begin{figure}
	\centering
    \includegraphics[width = \linewidth]{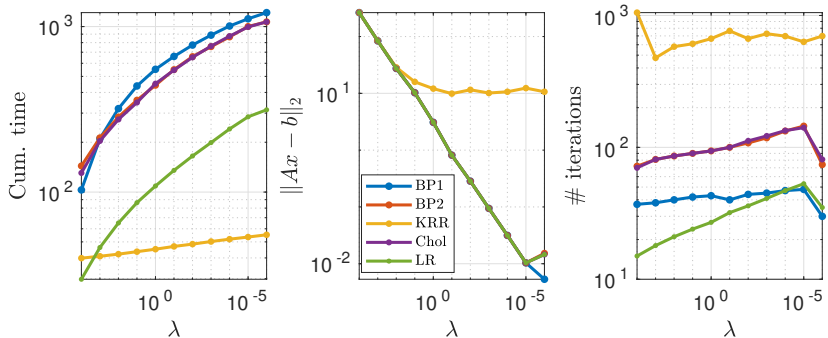}
    \caption{We compare various methods to solve the regularized overdetermined problem \eqref{eq:regLLSoverIter} for multiple values of $\lambda$.  The cum.  time ($y$-axis) displays the cumulative computing time necessary to loop through an extra value of $\lambda$.  BP1, BP2 and KRR are explained in Section \ref{sec:expcomp}. The final two methods, Chol and LR, refer to Algorithms \ref{alg:rankCholTikhonovover} and \ref{alg:randLowRankTikhonov}.  For BP 1, BP 2 and Chol we choose $s = 5n$ and a subsampled randomized DCT embedding , for KRR and LR we choose $s = 2\ceil{\text{sd}_{\lambda}(A)}$ and a Gaussian embedding. The matrix $A$ is $10^6\times 2500$ with singular values decaying from $10^4$ to $10^{-60}$ exponentially. }
    \label{fig:comp}
\end{figure}

The results of the comparison can be seen in Figure \ref{fig:comp}. Although KRR is the fastest method, we see that the instability due do ill-conditioning of $A$ results in inaccurate results.  This effect is avoided in the other methods as they precondition $A$ directly, instead of preconditioning the normal equations.  Additionally, we see the low-rank (LR) preconditioner from Algorithm \ref{alg:randLowRankTikhonov} significantly outperforms algorithms that do not exploit the low statistical dimension of the problem.  The difference between BP2 and the Cholesky-based Algorithm \ref{alg:rankCholTikhonovover} is not clear from Figure \ref{fig:comp}; we explore this further in Figure \ref{fig:compCholQR}.

\begin{figure}
\centering
\includegraphics[width=0.9\linewidth]{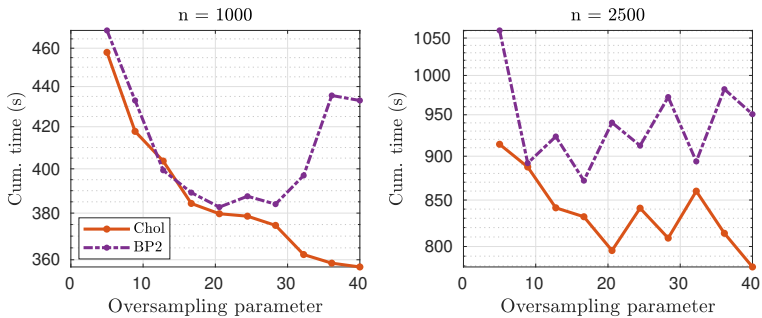}
\caption{We compare the performance of BP2 (see the start of Section \ref{sec:expcomp}) and Algorithm \ref{alg:rankCholTikhonovover}. The oversampling parameters refers to the size of the sketching dimension $s$ relative to $n$. That is, an oversampling parameter of 20 implies $s = 20n$. For each oversampling parameter, we compute the solution of \eqref{eq:regLLSoverIter} for fifteen values of $\lambda$. The cumulative time necessary to solve fifteen problems is displayed on the vertical axis. The matrix dimensions are $m=500000$ and $n=1000,2500$ respectively for the left and right plot.  The matrix has singular values exponentially decaying from $10^5$ to $10^{-5}$.}
\label{fig:compCholQR}
\end{figure}

Figure \ref{fig:compCholQR} shows how BP2 compares to Algorithm \ref{alg:rankCholTikhonovover}.  Both algorithms compute a preconditioner that is equivalent (in exact arithmetic), yet there is a difference in the computations. In the pre-processing step, BP2 only sketches $A$ to obtain $Y = XA\in\R^{s\times n}$ while Algorithm \ref{alg:rankCholTikhonovover} also computes $C= \YY$. However, for each value of $\lambda$ considered BP2 must compute the QR decomposition of an $(s+n)\times n$ matrix whereas in Algorithm \ref{alg:rankCholTikhonovover} we compute the Cholesky decomposition of an $n\times n$ matrix. Figure \ref{fig:compCholQR} shows clearly that for large values of $s$ this results in speed-ups. 

\subsection{Comparison of sketching matrices}
We next consider the effect of the type of sketching matrix used on the quality of the preconditioners obtained in Figure \ref{fig:timing1}. The figure shows that, although the sketching dimension of the SRTT matrices used is much larger than the sketching dimension of Gaussian embeddings, Gaussian embeddings produce higher quality preconditioners in the sense that $\kappa(BR^{-1})$ is smaller.  This also has as an effect that a smaller number of LSQR iterations is necessary.  The best choice will also depend on the number of regularization parameters $\lambda$ one wishes to consider.  If a large number of problems is to be solved, the reduction in the number of iterations due to a larger sketch size may improve the computational time sufficiently to compensate for the larger computing time in the sketching step.

\begin{figure}
    \hspace{-.7cm}
    \includegraphics[width = 1.1\linewidth]{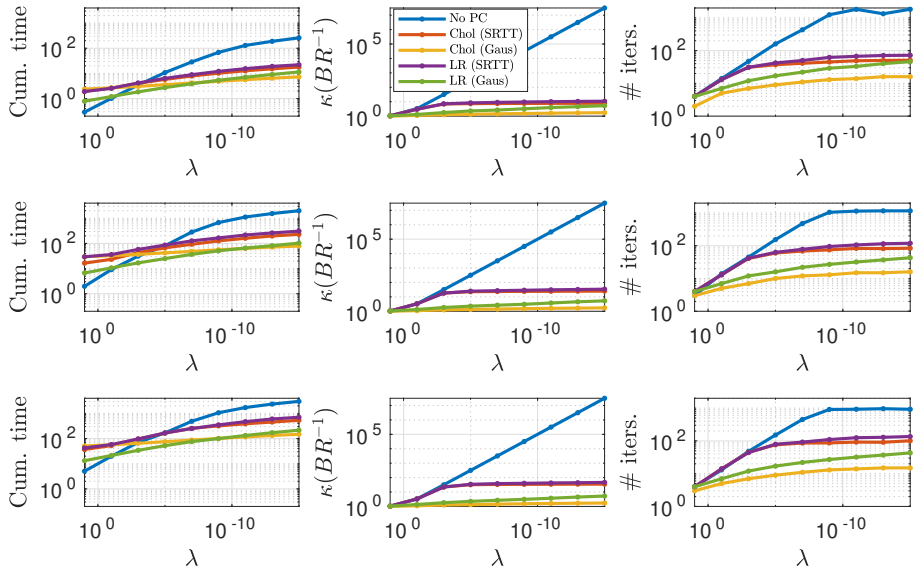}
    \caption{The rows correspond to matrices of sizes $m = 50000, 500000, 1000000$ respectively. In each case we choose $n = 1000$. The same problem is solved with 1) LSQR without preconditioning, 2) LSQR with the Cholesky preconditioner and an SRTT embedding matrix, 3) LSQR with the Cholesky preconditioner and a Gaussian embedding matrix, 4) LSQR with the low-rank preconditioner and an SRTT embedding matrix and 5) LSQR with the low rank preconditioner and a Gaussian embedding matrix. For the Cholesky preconditioner we choose $s = 10n$ and $s = 2n$ respectively for the SRTT and Gaussian embeddings, and for the low rank preconditioner we choose $s = \min(n,10\ceil{\text{sd}_{\lambda}(A)})$ and $s = 2\ceil{\text{sd}_{\lambda}(A)}$ for the SRTT and Gaussian embeddings. The singular values of $A$ decay exponentially from 1 to $10^{-50}$. The regularization parameters $\lambda_i$ range exponentially from 10 to $10^{-15}$. The timings shown are the cumulative timings for each additional value of $\lambda$ considered. We also plot the condition number of the (preconditioned) matrix and the number of iterations needed until convergence.}
    \label{fig:timing1}
\end{figure}

\subsection{L-curve to optimise the regularization parameter}\label{sec:lcurve}
Finally, we show how our algorithms can be used to compute the optimal regularization parameter and its corresponding solution. We particularly consider the method of L-curves \cite{Hansen2001TheProblems}, where for a set of regularization parameters the norm of the solution $\|x_{\lambda_i}\|_2$ is plotted against the residual $\|Ax_{\lambda_i} - b\|_2$ in a log-log plot. The plot should show an L-like shape, and the optimal regularization parameters correspond to the corner. In the example in Figure \ref{fig:lcurve} we see the optimal value is $\lambda^* = 10^{-3}$, and the computed solution was already computed. In addition to much better execution time for our algorithms as compared to LSQR without preconditioning, we see that the preconditioning allows us to compute more accurate solutions for larger $\lambda$.

\begin{figure}
    \centering
    \includegraphics[width = \linewidth]{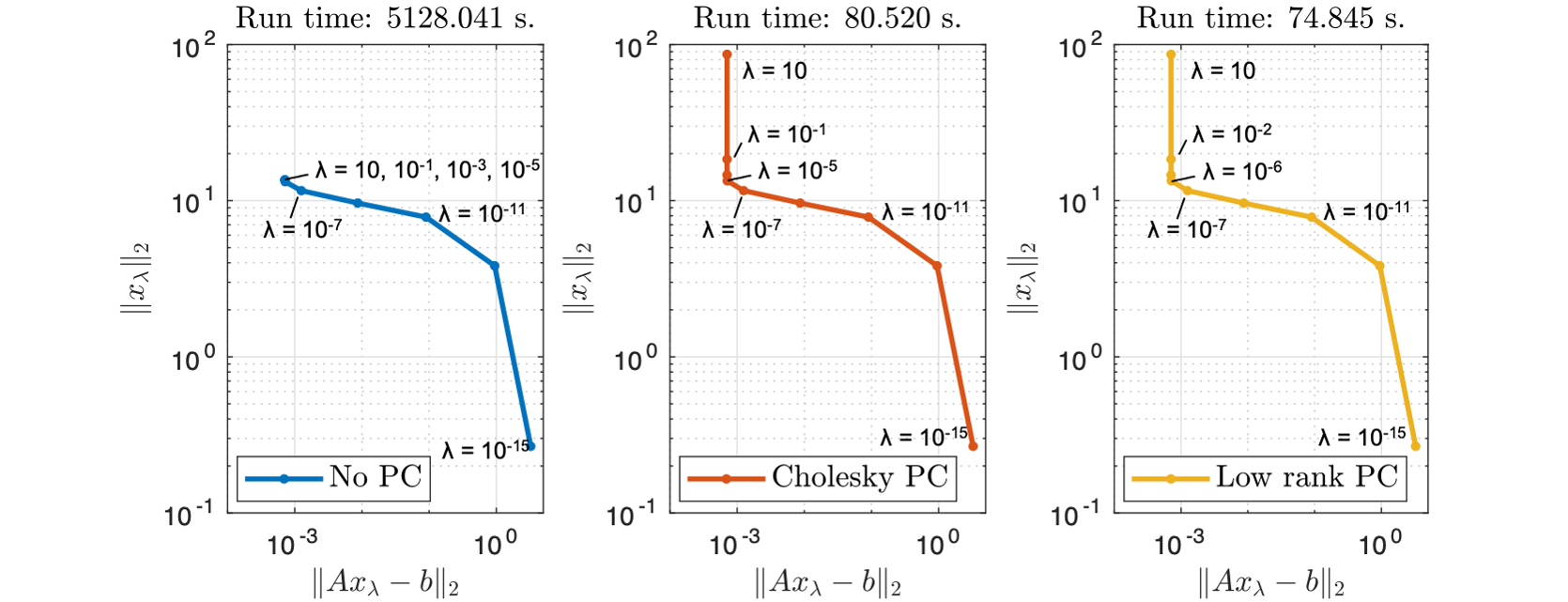}
    \caption{The L-curves \cite{Hansen2001TheProblems} and running times for the same problem solved with 1) LSQR without preconditioning (left),  2) LSQR with the Cholesky preconditioner (middle), and 3) LSQR with the low rank preconditioner (right).  For the Cholesky preconditioner we choose $s = 2n$, and for the low rank preconditioner we choose $s = 2\,\text{sd}_{\lambda}(A)$. The problem dimensions are $m=500000$ and $n = 2000$. The singular values of $A$ decay exponentially from 1 to $10^{-50}$. The regularization parameters $\lambda_i$ range exponentially from 10 to $10^{-15}$. }
    \label{fig:lcurve}
\end{figure}

\bibliography{references} 
\bibliographystyle{acm}

\end{document}